\documentclass[12pt,twoside]{amsart}
\input{includeNice3}
\usepackage{mabliautoref}
\usepackage{amssymb}
\usepackage{amscd}
\usepackage[abbrev,alphabetic]{amsrefs}
\usepackage{hyperref}
\usepackage[a4paper,margin=1in]{geometry}  
\usepackage{soul}

\usepackage{xypic}
\usepackage{tikz-cd}
\usepackage{caption}

\usepackage{url}
\RequirePackage{booktabs, multirow}
\RequirePackage{pgf}

\newcommand{\IB}[1]{\textcolor{purple}{(IB: #1)}}

\newcommand{\cO}{\mathcal{O}}

\newcommand{\cX}{\mathcal{X}}

\newcommand{\cS}{\mathcal{S}}

\title[Arithmetic and geometric deformations of 3-folds]
{Arithmetic and geometric deformations of 3-folds} 
\author{Fabio Bernasconi, Iacopo Brivio, and Stefano Filipazzi} 
\thanks{Part of this work was carried out during a visit of IB at EPFL.
He thanks the Chair of Algebraic Geometry at EPFL for the support provided. FB was partly supported by the grant $\#200021/169639$ from the
Swiss National Science Foundation, IB was supported by the National Center for Theoretical Sciences and a grant from the Ministry of Science and Technology,
grant number MOST-110-2123-M-002-005, SF was partly supported by ERC starting grant $\#804334$.
}

\address{EPFL SB MATH CAG
	MA  C3 625  (B\^atiment MA)
	Station 8
	CH-1015 Lausanne} 
\email{fabio.bernasconi@epfl.ch}

\address{National Center for Theoretical Sciences, Taipei, 106, Taiwan}
\email{ibrivio@ncts.ntu.edu.tw}

\address{EPFL SB MATH CAG
	MA  C3 625  (B\^atiment MA)
	Station 8
	CH-1015 Lausanne} 
\email{stefano.filipazzi@epfl.ch}

\DeclareMathOperator{\cent}{center}
\DeclareMathOperator{\cod}{codim}

\newcommand{\R}{\mathbb{R}}

\begin{document}
	
\maketitle

\begin{abstract}
We show that mixed-characteristic and equi-characteristic small deformations of 3-dimensional canonical (resp.~terminal) singularities with perfect residue field of characteristic $p>5$ are canonical (resp.~terminal). We discuss applications to arithmetic and geometric families of 3-dimensional Fano varieties and minimal models with canonical singularities.
Our results are contingent upon the existence of log resolutions for 4-folds.
\end{abstract}

\tableofcontents

\section{Introduction}

Singularities and their deformations have played a pivotal role in the development of birational geometry and the moduli theory of higher-dimensional algebraic varieties.
For instance, terminal singularities appear when running the Minimal Model Program (MMP, for short) for regular projective varieties, and canonical singularities appear when considering their canonical models.
For this reason, it is natural to ask whether these classes of singularities are stable under small deformations, i.e., deformations over spectra of discrete valuation rings (see \autoref{def: family_pairs}).

In equi-characteristic 0, the picture is nowadays quite complete. Inspired by the techniques used by Siu in the proof of invariance of plurigenera (\cite{Siu98}), Kawamata showed that small deformations of canonical singularities remain canonical using vanishing theorems for multiplier ideals sheaves (\cite{Kaw99}). By a similar strategy, one can show that the analogous result for terminal singularities also holds, as explained in \cite{Nak_ZDA}*{Chapter VI, \S~5}.
As for the larger class of klt (resp.~ log canonical) singularities, these are known to be stable under small deformations if the canonical divisor of the total space of the deformation is $\mathbb{Q}$-Cartier by plt (resp.~ log canonical) inversion of adjunction (\cites{km-book, kawakita}).
We refer to \cite{Ish18}*{Chapter 9} for a brief and clear proof of these results, using the MMP as developed in the seminal work \cite{BCHM10}. 
We remark that klt and log canonical singularities are also stable for small deformations under the weaker hypothesis that the canonical divisor of the generic fibre is $\mathbb{Q}$-Cartier: the case of surfaces has been proven in \cites{EV85, Ish86} and recently Sato and Takagi extended those results to higher dimension in \cite{SS22} using the theory of adjoint and non-log canonical ideal sheaves, which extend the notion of multiplier ideal sheaf to the non-$\bQ$-Gorenstein setting.
This is optimal as shown in \cite{Ish18}*{Example 9.1.8}.

In the case of terminal and canonical singularities, no requirement is needed on the canonical divisor of the total space of the family: indeed, as canonical singularities are Gorenstein in codimension 2 and they satisfy Serre's condition $(S_3)$ (\cite{Elk81}), an extension theorem due to Grothendieck (\cite{Gro68}*{XI.2.2}) implies that the canonical divisor of the family is $\mathbb{Q}$-Cartier; see also \cite{dFH11}.

In this article, we are interested in the deformation theory of the singularities appearing in birational geometry in positive and mixed characteristics.
As mentioned above, in characteristic 0 this topic can be studied via vanishing theorems and the MMP.
In positive or mixed characteristics, these results are in general false or unknown, respectively.
For this reason, we limit our discussion to the case of 3-dimensional singularities.
In this setup, the MMP has been recently developed in the series of articles \cites{HX15, Bir16, HW19, HW20, 7authors} and some vanishing theorems have been established in \cites{ABL20, HW19, BK20}, thus making it possible to approach the problem using a strategy analogous to \cite{Ish18}*{Chapter 9}.

To study deformations of 3-fold singularities over a DVR, we need to understand the singularities of the total space of the deformation, which is a 4-fold.
 In the recent article \cite{HW20}, the authors prove a birational version of the MMP for 4-folds in positive and mixed characteristics which in turn implies plt inversion of adjunction for $\mathbb{Q}$-factorial 4-folds with perfect residue field of characteristic $p>5$. 
As the $\mathbb{Q}$-factoriality assumption is too restrictive for studying general deformations of singularities, we use the recent results on the semi-stable MMP \cites{HW20, XX22} to establish inversion of adjunction for families of 3-folds.

Throughout this article we assume the existence of log resolutions for 4-folds (see \autoref{hyp}).

\begin{theorem}[cf.~\autoref{cor: plt-inv-adj}, \autoref{cor: inv-adjunction}] \label{t-inv-adjunction}
	Let $R$ be an excellent DVR with perfect residue field $k$ of characteristic $p>5$.
	Let $(X, \Delta) \to \Spec(R)$ be a family of 3-dimensional couples such that $K_X+\Delta$ is $\mathbb{Q}$-Cartier.
	If $(X_k, \Delta_k)$ is klt (resp.~ log canonical), then $(X, X_k+\Delta)$ is plt (resp.~ log canonical).
\end{theorem}

For a more general version, where the normalisation of the central fibre is log canonical we refer to \autoref{cor: inv-adjunction}.
Analogous statements have been recently proven for deformations of $F$-regular and $F$-pure singularities of arbitrary dimensions in \cites{MS21, ST21}.

Using the Cohen--Macaulay condition on the central fibre proven in \cite{BK20} (relying on the vanishing theorems proven in \cite{ABL20})
and the existence of terminalisations for families of klt pairs, we show that terminal and canonical 3-dimensional singularities are stable under small deformations. 
The $\mathbb{Q}$-factorial terminal case has been proven in \cite{HW20}*{Corollary 6.7}.

\begin{theorem} \label{thm: def_canonical}
	Let $R$ be an excellent DVR with perfect residue field $k$ of characteristic $p>5$.
	Let $(X, \Delta) \to \Spec(R)$ be a family of 3-dimensional couples.
	If $(X_k, \Delta_k)$ is canonical (resp.~ terminal), then $(X, \Delta)$ is canonical (resp.~ terminal) and the generic fibre $(X_{{K}}, \Delta_{{K}})$ is geometrically canonical (resp.~ terminal).
    Moreover, if $\Delta$ is a $\mathbb Q$-divisor with Weil index $I_1$ and the Cartier index of $K_{X_k}+\Delta_k$ is $I_2$, then the Cartier index of $K_X+\Delta$ is $\mathrm{lcm}(I_1,I_2)$.
    In particular, if $\Delta=0$, the Cartier index of $K_X$ coincides with the Cartier index of $K_{X_k}$.
\end{theorem} 

\begin{remark}
In characteristic 0, the analogue of \autoref{thm: def_canonical} was settled by de Fernex and Hacon under the additional assumption that $\lfloor \Delta_k \rfloor =0$; see \cite{dFH11}*{Proposition 3.5.(b)}.
The methods of \autoref{thm: def_canonical} also apply to study the case of deformations of canonical singularities that are not klt.
For the sake of completeness, we include this treatment as \autoref{theorem: char_0}.
\end{remark}

As applications of our results, we discuss geometric and arithmetic deformations of the outcomes of the 3-dimensional MMP.
In the case of Fano varieties, we show the following (cf.~\cite{dFH11}*{Proposition 3.8}):

\begin{corollary}[cf.~\autoref{cor: def_Fano}]
Let $S$ be an excellent Dedekind domain whose closed points have perfect residue field $p>5$ and let $X \to \Spec(S)$ be a projective flat morphism. 
If a closed fibre $X_s$ is a terminal weak Fano 3-fold, there exists an open set $U \subset \Spec(S)$ such that $X_u$ is a weak Fano with terminal singularities for all $u \in U$.
\end{corollary}

We prove similar statements for minimal models. 
As an interesting application, we verify the abundance conjecture for 3-dimensional minimal models which lift to characteristic 0.

\begin{corollary}[cf.~\autoref{cor: lift_MM_char0}]
    Let $X$ be a terminal 3-dimensional minimal model defined over a perfect field $k$ of characteristic $p>5$.
    If $X$ lifts to a mixed characteristic domain, then $X$ is a good minimal model and so is its lift.
    In particular, if $X$ is a terminal $K$-trivial 3-fold, then so is its lift.
\end{corollary}

It is unclear whether \autoref{thm: def_canonical} still holds for deformations of 3-folds if we have $\textup{char}(k) \in \{2,3,5\}$.
However, we show that the analogue of \autoref{thm: def_canonical} does not extend to families of 4-folds in characteristic 2.
In particular, we show that there exist families of canonical 4-folds in equi-characteristic 2 where the Cartier index of the canonical divisor jumps in family.
In this case, the central fibre is not $(S_3)$; see \autoref{s-example}.

\begin{theorem}[see \autoref{jump:perfect}] \label{t-example}
	Let $k$ be an algebraically closed field of characteristic $p=2$. Then there exists a family of 4-folds $X \to \mathbb{A}^1_k$ such that the Cartier index of $K_{X_0}$ is 1 and the Cartier index of $K_{X}$ is 2. 
\end{theorem}

We leave open the question of whether small deformations of canonical singularities are still canonical in higher dimensions. 

The organisation of the paper is as follows.
In \autoref{s-preliminaries}, we establish basic notations and we include a proof of deformation stability of canonical and terminal surface singularities, including the case of imperfect residue field.
In \autoref{s-MMP-input}, we collect the results on the MMP for semi-stable 4-folds over a DVR in order to establish \autoref{t-inv-adjunction}.
In \autoref{ss-deformations-canonical}, we use the result of \cite{BK20}
together with the semi-stable MMP to deduce the stability of terminal and canonical 3-dimensional singularities under deformations. 
We then present the applications to families of weak Fano 3-folds and minimal models over positive and mixed characteristics DVR and Dedekind domains in \autoref{ss-deformations-Fano}.
In \autoref{s-example}, we use the irregular del Pezzo surface with canonical singularities over an imperfect field of characteristic $p=2$ of Schr\"{o}er in \cite{Sch07} and a spreading out argument to show \autoref{t-example}.

\subsection*{Acknowledgements}
We would like to thank C.D. Hacon and J. Witaszek for feedback on this work, and the anonymous referee for useful comments and suggestions that improved the clarity of this work.

\section{Preliminaries} \label{s-preliminaries}

\subsection{Notation}

\begin{enumerate}
    \item Throughout the article, $p>0$ denotes a fixed prime number.
    \item All rings $S$ are assumed to be Noetherian, excellent, of finite Krull dimension, and admitting a dualising complex $\omega_{S}^{\bullet}$ normalised as in \cite{7authors}*{Section 2.1}. 
    We require our rings to be excellent, as this notion is well-behaved with respect to localisation, completion, regularity, and it is the natural class for the problem of resolution of singularities. For an overview of the importance of the excellence condition, we refer to \cite{Ray14}.
    \item Given a field $k$, we denote by $\overline{k}$ (resp.~ $k^{\sep}$) an algebraic closure (resp. separable closure). 
    We denote by $k^{1/p^\infty}$ the perfect closure of $k$.
    \item We denote by $(A, \mathfrak{m},\mathfrak{k})$ a local ring $A$ with maximal ideal $\mathfrak{m}$, and residue field $\mathfrak{k}$. Throughout the article, we suppose $\characteristic(\mathfrak{k})=p>0$. 
    \item If $R$ is a DVR, we denote by $K$ its fraction field and by $k$ its residue field.
    If $f \colon X \to \Spec(R)$ is a flat morphism, we denote the special (resp. generic) fibre by $X_k \coloneqq X \times_{\Spec(R)} \Spec(k)$ (resp. $X_K \coloneqq X \times_{\Spec(R)} \Spec(K)$).
    \item We say $A$ is \emph{Gorenstein} if it is Cohen--Macaulay and the dualising sheaf $\omega_A \coloneqq H^{-\dim(A)}(\omega_A^{\bullet})$ is an invertible sheaf. 
    Note that for integrally closed domains of dimension $\geq 3$, the Gorenstein condition is stronger than just requiring the canonical divisor to be Cartier.

    \item  We say $(X, \Delta)$ is a \emph{couple} if $X$ is an equidimensional normal Noetherian excellent integral scheme admitting a dualizing complex, and $\Delta$ is an effective $\mathbb{R}$-divisor with coefficients in $[0, 1]$.
    We say a couple $(X, \Delta)$ is a \emph{pair} if $K_X+\Delta$ is $\mathbb{R}$-Cartier.
    
    \item Given a normal scheme $X$ and a $\mathbb{Q}$-Cartier $\mathbb{Q}$-divisor $D$ on $X$, the \emph{Cartier index} of $D$ is the smallest integer $n \in \mathbb{Z}_{>0}$ such that $nD$ is an integral Cartier divisor.
    If $D$ is a $\mathbb{Q}$-divisor, the \emph{Weil index} of $D$ is the smallest integer $n \in \mathbb{Z}_{>0}$ such that $nD$ is an integral divisor.
    \item Given an $\mathbb{R}$-Cartier divisor $L$ on a normal projective variety $X$ over a field, we denote by $\kappa(L)$ its Kodaira dimension. If $L$ is nef, we denote by $\nu(L)$ its numerical dimension.

    \item Let $\pi\colon Y\to X$ be a birational morphism between integral normal schemes. We denote by $\Ex(\pi)$ the \textit{exceptional locus of} $\pi$, i.e. the locus of points $x\in X$ such that $\pi$ is not an isomorphism in a neighborhood of $x$.
\end{enumerate}

\subsection{Singularities, models, and families}

Given a pair $(X,\Delta)$ and a proper birational morphism $\pi \colon Y \to (X,\Delta)$ we can write
$$K_Y+\pi_*^{-1}\Delta=\pi^*(K_X+\Delta) + \sum_i a(E_i, X, \Delta)E_i, $$
for some $a(E_i, X, \Delta) \in \mathbb{R}$, where $E_i$ runs through the prime components of the $\pi$-exceptional divisors.
The number $a({E_i},X, \Delta)$ is called the \emph{discrepancy} of $E_i$ with respect to $(X,\Delta)$ and it does not depend on the model $Y$ on which $E_i$ appears.
For a thorough discussion of the singularities of the MMP for excellent rings, we refer to \cite{kk-singbook} and \cite{7authors}.

\begin{definition} \label{def: canonical_terminal_sing}
A pair $(X,\Delta)$ is said to be \emph{canonical} (resp. \emph{terminal}) if for all proper birational morphism $\pi \colon Y \to X$ where $Y$ is normal and every $\pi$-exceptional divisor $E$, we have $a(E, X, \Delta) \geq 0$ (resp. $a(E, X, \Delta) > 0$). 

If $X$ is a defined over a field $k$, we say $(X, \Delta)$ is \emph{geometrically canonical} (resp. terminal) over $k$ if $(X_{\overline{k}}, \Delta_{\overline{k}})$ is canonical (resp. terminal).
\end{definition}

\begin{remark}
If $(X,\Delta)$ is canonical, then it is plt. In particular, if $(X,\Delta)$ is canonical then it is klt if and only if $\lfloor{\Delta \rfloor}=0$.
On the other hand, if $(X,\Delta)$ is terminal, it is necessarily klt.
\end{remark}

We recall the definitions of the birational modifications we will need in this article.

\begin{definition}\label{def: dlt-modification}
    Let $(X,\Delta)$ be a couple.
    We say $\pi \colon Y \to (X,\Delta)$ is a \emph{dlt modification} if $\pi$ is a projective birational morphism, $(Y, \pi_*^{-1}\Delta+\sum_i E_i)$ is $\mathbb Q$-factorial dlt and $K_Y+\pi_*^{-1}\Delta+\sum_i E_i$ is $\pi$-nef, where the $E_i$ are the irreducible components of the divisorial part of $\Ex(\pi)$.
\end{definition}

\begin{remark} \label{rmk: inv_of_adj}
    In the context of \autoref{def: dlt-modification}, further assume that $K_X+\Delta$ is $\mathbb \R$-Cartier and write $K_Y+\Delta_Y=\pi^*(K_X+\Delta)$.
    Then, for every $x \in X$, either $\pi^{-1}(x)$ is contained in $\Supp(\pi_*^{-1}\Delta+\sum_i E_i-\Delta_Y)$ or it is disjoint from it; see, e.g., \cite{7authors}*{proof of Corollary 9.21}.
\end{remark}

\begin{definition}\label{def:terminalisation}
    Let $(X,\Delta)$ be a klt pair. 
    We say $\pi \colon Y \to X$ is a \emph{terminalisation} if $(Y,\Delta_Y)$ is terminal pair, where $K_Y+\Delta_Y=\pi^*(K_X+\Delta)$.
\end{definition}

We introduce the notion of families of couples over a 1-dimensional regular base we use in this work.

\begin{definition} \label{def: family_pairs}
Let $C$ be a regular excellent 1-dimensional scheme and $n>0$.  
We say $f \colon (X, \Delta) \to C$ is a \emph{family of $n$-dimensional couples} if
\begin{enumerate}
    \item $(X,\Delta)$ is a couple;
    \item $f \colon X \to C$ is a flat closed morphism of relative dimension $n$ and essentially of finite type;
    \item for each irreducible component $B$ of $\Supp(\Delta)$, the restriction $f|_B \colon B \to C$ is flat;
    \item for every $c \in C$, the base change $(X_c,\Delta_c)$ is a couple of dimension $n$ where $\Delta_c$ is the divisorial restriction of $\Delta$ as in \cite{k-moduli}*{Section 2.1}.
\end{enumerate}
If $\Delta =0$, we say that $f \colon X \rightarrow C$ is a {\it family of $n$-dimensional varieties}. When $C=\Spec(R)$, where $R$ is a DVR with residue field $k$, we will sometimes refer to $f \colon (X, \Delta) \to C$ as a \textit{small deformation of $(X_k,\Delta_k)$}.
\end{definition}


\begin{remark}
    Note that if $f \colon (X, \Delta) \to C$ is a family of couples with geometrically reduced fibres, then it is also a \textit{family of pairs} as defined in \cite{k-moduli}*{Definition 2.2}. 
    We chose to distinguish between the notions of couple and pair, since the latter requires the log canonical divisor of the total space to be $\bR$-Cartier, while an important step in the proof of \autoref{thm: def_canonical} is showing that $(X_K,\Delta_K)$ and $(X,\Delta)$ are actual pairs.
\end{remark}

\begin{remark} \label{rem: normality_family}
    As  $(X_k, \Delta_k)$ is a couple, by definition $X_k$ is normal. 
    As $X_k$ is a Cartier divisor, $X$ has property $(R_1)$ by \cite{stacks-project}*{\href{https://stacks.math.columbia.edu/tag/00NU}{Tag 00NU}} and $(S_3)$ by \cite{kk-singbook}*{Corollary 2.61}.
\end{remark}

\begin{remark}
    Recall that even if all fibres are pairs, it is not necessarily true that $K_X+\Delta$ is $\mathbb{R}$-Cartier. See \cite{Ish18}*{Example 9.1.7}.
\end{remark}

We will use the following lemma several times, see also \cite{k-moduli}*{Proposition 2.15}.

\begin{lemma}\label{lem: lemma_trick}
Let $R$ be an excellent DVR, and let $f \colon (X,\Delta) \rightarrow \Spec(R)$ be a family of couples such that $(X,X_k+\Delta)$ is a log canonical pair.
Then, if $E$ is an exceptional divisorial valuation over $X$ such that $\cent_X(E)\subset X_k$, then $a(E,X,\Delta) \geq 0$.
Furthermore, if $(X,X_k+\Delta)$ is plt, we have $a(E,X,\Delta) > 0$.
\end{lemma}

\begin{proof}
Let $\pi \colon Y \rightarrow X$ be a proper birational morphism from a normal variety $Y$ such that $E$ is a divisor on $Y$.
Let $m_E$ denote the coefficient of $E$ in $\pi^*(X_k)$.
Since $X_k$ is a Cartier divisor and $\cent_X(E) \subset X_k$, $m_E$ is a positive integer.
Since $(X,X_k+\Delta)$ is log canonical, we obtain
\[
-1 \leq a(E,X,X_k+\Delta) = a(E,X,\Delta)-m_E,
\]
and therefore we have
\[
a(E,X,\Delta) \geq m_E-1 \geq 0.
\]
We observe that the equality case $a(E,X,\Delta)=0$ forces $E$ to be a log canonical place of $(X,X_k+\Delta)$.
Thus, if we further assume that $(X,X_k+\Delta)$ is plt, as $X_k$ is the only log canonical place of $(X,X_k+\Delta)$, we obtain the strict inequality $a(E,X,\Delta) > 0$.
\end{proof}

We recall the definition of Fano varieties and minimal models.

\begin{definition}\label{def: Fano-minimal-models}
    Let $(X,\Delta)$ be a klt pair, projective over a base $R$.
    If $-(K_X+\Delta)$ is big and nef, we say $(X,\Delta)$ is a \textit{weak log Fano} pair.
    If $\Delta=0$, we say it is \textit{weak Fano}.
    
    If $K_X+\Delta$ is nef, we say $(X,\Delta)$ is a \textit{minimal model}.
    Furthermore, if $K_X+\Delta$ is semi-ample, we say it is a {\it good minimal model}.
\end{definition}

\subsection{Deformations of canonical surface singularities}

We prove that small deformations of terminal and canonical singularities for surfaces stay in the same class. 
This is well-known in the case of a perfect residue field but for sake of completeness, we include a version including the imperfect case.

We recall the definition of \emph{multiplicity} for an effective divisor on a regular scheme $X$.

\begin{definition}[{\cite{Fulton_IT}*{Section 4.3}}]
    Let $(X,x)$ be the spectrum of a regular local ring $(A, \mathfrak{m},\mathfrak{k})$ of dimension $d$, and let $D$ be a Cartier divisor with defining equation $f\in A$. 
    Let $\pi\colon Y \to X$ be the blow-up of the maximal ideal $\mathfrak{m}$ with exceptional divisor $E\simeq\mathbb{P}^{d-1}_{\mathfrak{k}}$. The \textit{multiplicity of $D$ at $x$} is denoted by $\mult_x(D)$ and defined, equivalently, as:
    \begin{itemize}
        \item the largest integer $\mu\geq 0$ such that $f\in\mathfrak{m}^\mu$; or
        \item the coefficient of $E$ in $\pi^*D-\pi_*^{-1}D$.
    \end{itemize}
    If $D$ is an $\mathbb R$-Cartier divisor and $D=\sum a_i D_i$ where each $D_i$ is a Cartier divisor and , we define $\mult_x(D)\coloneqq \sum a_i \mult_x(D_i)$.
\end{definition}

We recall the following well-known fact on the behaviour of multiplicities in families.

\begin{lemma}\label{lem: multiplicity}
Let $R$ be an excellent DVR and let $X \to \Spec(R)$ be a flat morphism such that $X_k$ is regular.
If $D$ is a $\mathbb{R}$-Cartier divisor on $X$, then $\mult_x(D) \leq \mult_x(D|_{X_k})$ for any point $x \in X_k$.
\end{lemma}

\begin{proof}
    By \cite{stacks-project}*{\href{https://stacks.math.columbia.edu/tag/00NU}{Tag 00NU}}, $X$ is regular in an open neighbourhood of $x$.
    By the definition of multiplicity for an $\mathbb R$-Cartier divisor, we may assume that $D$ is Cartier.
    Without loss of generality, we may assume that $X$ is the spectrum of a regular local $R$-algebra $(A,\mathfrak{m},\mathfrak{k})$, and $D=(f=0)$ is Cartier.
    Denote by $A_k$ the coordinate ring of $X_k$, by $\mathfrak{m}_k=\textup{Im}\left( \mathfrak{m}\otimes_R k\to A\otimes_R k \right)$ its maximal ideal, and by $\mu$ the multiplicity of $D$ at $x$. As $f\in\mathfrak{m}^{\mu}$, then $f\vert_{X_k}\in\mathfrak{m}_k^{\mu}$.
\end{proof}

\begin{proposition}\label{prop: def_can_surfaces}
    Let $R$ be an excellent DVR and let $f \colon (X,\Delta) \to \Spec(R)$ be a family of 2-dimensional couples.
    If $(X_k, \Delta_k)$ is canonical (resp.~ terminal), then so is $(X,\Delta)$.
    In both cases, the total space $X$ has Gorenstein singularities.
    Moreover, if $\Delta$ is a $\mathbb Q$-divisor with Weil index $I_1$, then the Cartier index of $K_X+\Delta$ is $I_1$.
\end{proposition}

\begin{proof}
    Without loss of generality, we can localise at a closed point $x \in X_k$.
    We first suppose that $\Delta=0$.
    Since $X_k$ is canonical, then $X_k$ is Gorenstein by \cite{kk-singbook}*{Theorem 2.29} and therefore $X$ is Gorenstein by \cite{stacks-project}*{\href{https://stacks.math.columbia.edu/tag/0BJJ}{Tag 0BJJ}}.
    By \autoref{rem: normality_family}, $X$ is normal.
    By inversion of adjunction \cite{k-notQfact}*{Corollary 10}, the pair $(X, X_k)$ is plt, and thus $X$ is klt. 
    As $X$ is also Gorenstein, we conclude $X$ has canonical singularities.
    If we further assume that $X_k$ is terminal, by \cite{kk-singbook}*{Theorem 2.29} $X_k$ is regular. Thus $X$ is regular by \cite{stacks-project}*{\href{https://stacks.math.columbia.edu/tag/00NU}{Tag 00NU}} and thus it is terminal.

    Suppose now $\Delta \neq 0$. 
    By \cite{kk-singbook}*{Theorem 2.29}, we have that $X_k$ is regular and $\mult_x(\Delta_k)\leq 1$. 
    Hence, by \cite{stacks-project}*{\href{https://stacks.math.columbia.edu/tag/00NU}{Tag 00NU}}, we have that $X$ is also regular.
    Suppose by contradiction there exists an exceptional divisor $E$ for $(X, \Delta)$ such that $a(E,X, \Delta) <0$.  
    By the previous paragraph, we may assume that $\cent_X(E) \subset \Supp(\Delta)$ and by \autoref{lem: lemma_trick}, $\cent_X(E)$ dominates $\Spec(R)$.
    In particular, we have that $a(E, X, \Delta)=a(E_K, X_K, \Delta_K)$ and thus it is sufficient to show that $(X_K, \Delta_K)$ is canonical.
    As $y \mapsto \mult_y(\Delta)$ is upper semi-continuous, we know that $\mult_{y}(\Delta) \leq \mult_x(\Delta)$ for every point $y$ such that $x$ belongs to the closure $\overline{y}$. 
    Thus, by \autoref{lem: multiplicity}, we deduce that $\mult_y(\Delta_K) \leq 1$ for every $y \in X_K$ and hence $(X_K, \Delta_K)$ is canonical by \cite{kk-singbook}*{Theorem 2.29}.
    
    We now show that if $(X_k, \Delta_k)$ is terminal, then so is $(X, \Delta)$.
    Assume by contradiction that this is not the case and let $E$ be an exceptional divisorial valuation such that $a(E, X, \Delta) \leq 0$.
    For every $0<\varepsilon \ll 1$, as $(X_k, (1+\varepsilon) \Delta_k)$ is terminal, then the pair $(X, (1+\varepsilon)\Delta)$ is canonical by what just shown.
    As we reduced to the case where $\cent_X(E) \subset \Supp(\Delta)$, we obtain that $a(E, X, (1+\varepsilon)\Delta) < a(E, X, \Delta) \leq 0$, reaching the desired contradiction.

    To conclude, note that if $x\in\Supp(\Delta)$ then $X$ is regular at $x$, hence the Cartier and Weil indices of $K_X+\Delta$ coincide. If $x\notin\Supp(\Delta)$ then we conclude as 2-dimensional canonical singularities are Gorenstein.
\end{proof}

\begin{corollary}\label{cor: def_canonical}
    Let $R$ be an excellent DVR and let $f \colon (X,\Delta) \to \Spec(R)$ be a family of 2-dimensional couples.
    If $(X_k, \Delta_k)$ is geometrically canonical (resp.~ geometrically terminal), then so is $(X_K,\Delta_K)$.
\end{corollary}

\begin{proof}
    By \autoref{prop: def_can_surfaces} we are only left to check that $(X_{\overline{K}}, \Delta_{\overline{K}})$ is canonical. 
    If $\characteristic(K)=0$, we conclude by \cite{kk-singbook}*{Proposition 2.15} as every field extension is separable and thus \'etale.
    If $\characteristic(K)=p>0$, again by \cite{kk-singbook}*{Proposition 2.15} it is sufficient to check that $(X_{K^{{1/p^e}}}, \Delta_{K^{1/p^e}})$ is canonical for all $e \geq 0.$
    Consider the base change $X^{(e)}\coloneqq X \times_R R^{1/p^e}$ via the Frobenius morphism $\Spec(R^{1/p^e}) \to \Spec(R)$. Note that $(X^{(e)}, \Delta^{(e)})\to \Spec(R^{1/p^e})$ is still a family of couples, since the Frobenius morphism is a universal homeomorphism.
    We have a natural isomorphism $((X^{(e)})_k, (\Delta^{(e)})_k) \simeq (X_{k^{1/p^e}}, \Delta_{k^{1/p^e}})$ which is canonical by hypothesis. 
    Therefore by \autoref{prop: def_can_surfaces} the pair $(X^{(e)}, \Delta^{(e)})$ is canonical and thus $(X_{K^{1/p^e}}, \Delta_{K^{1/p^e}})$ is canonical as well by localisation.
\end{proof}

\section{Running a birational MMP over a DVR in dimension 4} \label{s-MMP-input}
 
Throughout this section, we fix $R$ to be an excellent DVR with perfect residue field of characteristic $p>5$.
We aim to prove a generalisation of \cite{HW20}*{Theorem 4.1} to not necessarily $\mathbb{Q}$-factorial bases.

\begin{hypothesis}
\label{hyp}
Following \cites{HW20, XX22}, we assume that log resolutions of couples in dimension 4 exist and are given by a sequence of blow-ups along the non-snc locus.
\end{hypothesis}

We now prove an MMP statement needed to prove inversion of adjunction and the existence of terminalisations.

\begin{proposition}\label{t-compactmmp}
    Let $X$ be an integral normal scheme of dimension 4, dominant and quasi-projective over $\Spec(R)$.
    Let $f\colon Y\to X$ be a projective log resolution of $(X,X_{k,\red})$.
    Assume that $(Y, B)$ is a dlt pair such that $\lfloor B \rfloor = Y_{k, \red}+\Gamma$, where $\Gamma$ is some reduced divisor (possibly empty).
    Then we can run a $(K_Y+B)$-MMP over $X$ with scaling of an ample divisor, which terminates with a relatively minimal model.
\end{proposition}
\begin{proof}
      We define $\Xi \coloneqq B - \lfloor B \rfloor$.
      If $X$ is not projective over $R$, we can take a compactification $\overline{X}\supset X$ such that $\overline{X}\to\Spec (R)$ is projective and $\overline{X}$ is normal.
      Similarly, after possibly blowing up more, by \autoref{hyp} we can find a regular compactification $\overline{Y}$ of $Y$, projective over $R$, fitting in the following commutative diagram
    \begin{center}
        \begin{tikzcd}
            Y \arrow[rr] \arrow[d,"f",swap] & & \overline{Y} \arrow[d,"\overline{f}"] \\
            X \arrow[rr] \arrow[dr] & & \overline{X} \arrow[dl]\\
            & \Spec (R), &
        \end{tikzcd}
    \end{center}
    where the square is Cartesian and $\overline{f}$ is a log resolution of $(\overline{X},\overline{X}_{k,\red})$, such that $(\overline{Y},\overline{Y}_{k, \red}+ \overline{\Gamma} + \overline{\Xi})$ is dlt.
    Here, $\overline{\Gamma}$ and $\overline{\Xi}$ denote the closures of $\Gamma$ and $\Xi$ in $\overline{Y}$, respectively.
    Also, we define $\overline{B} \coloneqq \overline{Y}_{k,\red}+\overline{\Gamma}+\overline{\Xi}$.
    
    We now verify that it is possible to run a $(K_{\overline{Y}}+\overline{B})$-MMP over $\overline{X}$ with scaling of an ample divisor.
    As $\overline{NE}(\overline{Y}/\overline{X}) \hookrightarrow \overline{NE}(\overline{Y}/R)$, the cone theorem is valid by \cite{XX22}*{Proposition 4.3}.
    As for the existence of the contraction of the MMP, the necessary base-point free theorem is proven in \cite{XX22}*{Proposition 4.4 and Remark 4.5}.
    Finally, the existence of flips is proven in \cite{XX22}*{Theorem 4.6} and thus we are left to show termination.
    By special termination \cite{XX22}*{Theorem 3.4 and Theorem 3.13}, the flipping and flipped loci are eventually disjoint from the strict transforms of $\overline{\Gamma}$ and $\overline{Y}_{k,\red}$.
    By \cite{XX22}*{Proposition 4.3}, at every step of the MMP, every curve class that is vertical over $\Spec (R)$ can be represented by a 1-cycle that is supported on the special fibre.
    Thus, the MMP has to terminate, since by special termination it is eventually disjoint from the special fibre.
\end{proof}

\begin{corollary} \label{cor: existence-dlt-modificatin}
    Let $(X,\Delta)\to\Spec(R)$ be a family of 3-dimensional couples that is quasi-projective over $R$. 
    Then there exists a dlt modification $\pi \colon Y \to (X,X_k+\Delta)$ as in \autoref{def: dlt-modification}.
\end{corollary}

\begin{proof}
    Let $f \colon W \to (X,\Delta)$ be a log resolution of $(X, X_k+\Delta)$. 
    We can suppose that the exceptional locus supports a relatively ample divisor by \cite{KW21}*{Theorem 1}, which implies that the exceptional locus $\Ex(f)$ is divisorial.
    Then we can run a $(K_W+f_*^{-1}(X_k+\Delta)+\Ex(f))$-MMP with scaling as in \autoref{t-compactmmp} to reach a dlt modification.
\end{proof}

We now prove various versions of inversion of adjunction for 4-folds.

\begin{corollary} \label{cor: plt-inv-adj}
    Let $(X,\Delta)\to\Spec(R)$ be a family of 3-dimensional couples such that $K_X+\Delta$ is $\mathbb{R}$-Cartier.
    If $(X_k, \Delta_k)$ is klt, then $(X, X_k+\Delta)$ is plt, and $(X_{\overline{K}},\Delta_{\overline{K}})$ is klt.
\end{corollary}

\begin{proof}
    Since the statement is local, we can suppose $X$ to be affine and of finite type over $R$.
    The same proof of \cite{HW20}*{Corollary 4.9}, using \autoref{cor: existence-dlt-modificatin}, yields plt-ness of $(X,X_k+\Delta)$. The same argument as in \autoref{cor: def_canonical} yields that $(X_{\overline{K}},\Delta_{\overline{K}})$ is klt. 
\end{proof}

\begin{corollary}\label{cor: inv-adjunction}
    Let $(X, \Delta)$ be a pair of dimension 4 and let $X \to\Spec(R)$ be a flat and closed morphism with reduced connected fibres.
    Suppose that $\Supp(\Delta)$ does not contain any irreducible component of $X_k$ and none of the irreducible components of $X_k \cap \Supp(\Delta)$ is contained in $\Sing(X_k)$.
    If the normalisation $(X_k^\nu, \Delta_k^\nu+D)$ is log canonical, where $D$ is the divisorial part of the subscheme cut out by the conductor ideal, then $(X, X_k+\Delta)$ is log canonical, and so is $(X_{\overline{K}},\Delta_{\overline{K}})$.
\end{corollary}

\begin{proof}   
    Since the statement is local, we can suppose $X$ to be affine and of finite type over $R$.
    Note that asking the divisorial part of the conductor to be reduced is equivalent to the fact that the singularities of $X_k$ are at worst nodal in codimension 1 by \cite{k-moduli}*{Lemma 11.16}.
    Consider a dlt modification $\pi \colon Y \rightarrow X$ of $(X,X_k+\Delta)$ as in \autoref{cor: existence-dlt-modificatin} with exceptional divisors $E_i$.
    Then, $(Y,\pi^{-1}_*(X_k+\Delta)+\sum E_i)$ is a dlt pair such that $K_Y+\pi^{-1}_*(X_k+\Delta)+\sum E_i$ is nef over $X$.
    Furthermore, we may write
    \begin{equation}\label{eq: crepant_pullback}
    \pi^{*}(K_X+X_k+\Delta)=K_Y+\pi^{-1}_*(X_k+\Delta)+\sum E_i+F,
    \end{equation}
    where $F$ is $\pi$-exceptional and effective by the negativity lemma \cite{7authors}*{Lemma 2.16}. 
    In particular, $F=0$ if and only if $(X,X_k+\Delta)$ is log canonical.
    Assume by contradiction $F \neq 0$.
    By \autoref{eq: crepant_pullback}, we have
    \[
    -F \equiv_X K_Y+\pi^{-1}_*(X_k+\Delta)+\sum E_i.
    \]
    In particular, $-F$ is $\pi$-nef, anti-effective, and $\pi$-exceptional.

    By construction, $\pi$ has connected fibres.
    Furthermore, for every $\pi$-exceptional divisor $E$, we have $\cent_X(E) \cap X_k \neq \emptyset$.
    Thus, by \autoref{rmk: inv_of_adj}, $\Supp(F) \cap \pi^{-1}_*X_k \neq \emptyset$.
    Since $Y$ is $\mathbb Q$-factorial, this intersection is a divisor in $\pi_*^{-1}X_k$.
    Then, by adjunction, we obtain that $(X_k^\nu, \Delta_k^\nu+D)$ is not log canonical, thus reaching the sought contradiction.
    The same argument as in \autoref{cor: def_canonical} yields that $(X_{\overline{K}},\Delta_{\overline{K}})$ is log canonical.
\end{proof}

We now show the existence of terminalisation for klt pairs.

\begin{corollary} \label{cor: existence_terminalisation}
Let $(X, \Delta) \to \Spec(R)$ be a quasi-projective family of couples such that $K_X + \Delta$ is $\mathbb{R}$-Cartier.
If $(X_k, \Delta_k)$ is a 3-dimensional klt pair, then there exists a terminalisation $Z \to X$ of $(X,\Delta)$ such that $Z_k$ is irreducible.
\end{corollary}

\begin{proof}
By inversion of adjunction, i.e., \autoref{cor: plt-inv-adj}, we know the pair $(X, X_k+\Delta)$ is plt.
Let $f\colon Y\to X$ be a log resolution of $(X,X_k+\Delta)$ and write $K_Y+D=f^*(K_X+\Delta)+F$, where $D$ and $F$ are effective $\mathbb{R}$-divisors without common components.
Since $(X, \Delta)$ is klt, up to a further sequence of blow-ups, we may assume that no distinct components of $\Supp(D)$ intersect.
In particular, we may assume that $(Y,D)$ is terminal.
Since $(Y,D)$ is terminal and $Y$ is regular, letting $D^h$ be the part of $D$ which is horizontal over $\Spec(R)$, we obtain that $(Y,Y_{k,\red}+D^h)$ is a dlt pair such that $(Y_K, D_K)$ is terminal.
By \autoref{t-compactmmp}, we can run a $(K_Y+Y_{k,\red}+D^h)$-MMP over $X$, which terminates with a relatively minimal model $g \colon (Z,Z_{k,\red}+D^h_Z) \to X$, where $D^h_Z$ is the push-forward of $D^h$ under the birational contraction $Y \dashrightarrow Z$.
By construction, we have that $(Z,Z_{k,\red}+D^h_Z)$ is dlt and $g_*(Z_{k,\red}+D^h_Z)=X_k+\Delta$. 
We observe $K_{Y_K}+D_K \equiv_{X_K} F_K$.
Then, $Y_K \dashrightarrow Z_K$ is $F_K$-negative.
In particular, $Y_K \dashrightarrow Z_K$ is an isomorphism away from $\Supp(F_K)$.
Therefore, since $D$ and $F$ share no common components, no irreducible component of $D_K$ is contracted by $Y_K \dashrightarrow Z_K$.
Since the MMP terminates with a relatively minimal model, the strict transform of $F_K$ on $Z_K$ is relatively nef over $X_K$.
Then, the negativity lemma \cite{7authors}*{Lemma 2.16} implies that all irreducible components of $F_K$ are contracted on $Z_K$.
Therefore, $(Z_K, D_{Z,K})$ is a terminalisation of $(X_K, \Delta_K)$.

Then, the negativity lemma \cite{7authors}*{Lemma 2.16} yields
\begin{equation}\label{e-neg}
K_Z+Z_{k,\red}+D_Z^h=g^*(K_X+X_k+\Delta)-E,
\end{equation}
where $E$ is an effective $g$-exceptional divisor. 
As $(Z_K, D_{Z,K})$ is a terminalisation, we know $\cent_X(E) \subset X_k$.
Now we argue that $Z_k$ is irreducible: as $E$ is supported on the special fibre, this will also imply $E=0$.
Let $P$ be a $g$-exceptional prime component of $Z_{k,\red}$. Then by plt-ness we know that $a(P, X, X_k + \Delta) > -1$. 
By \autoref{e-neg}, we also have $a(P, X, X_k + \Delta) \leq a(P, Z, Z_{k, \red} +D_Z^h)=-1$, reaching a contradiction.
Notice that, since $Z_k$ is irreducible, it is a Cartier divisor, and $X_k$ is reduced, it follows that also $Z_k$ is reduced.

By what we have just shown, we have that $D_Z^h=D_Z$ and $K_Z + D_Z^h = g^*(K_X+ \Delta)$. 
We are left to show that $(Z, D_Z)$ is terminal. 
Assume by contradiction it is not the case and there exists an exceptional divisorial valuation $Q$ such that $a(Q, Z, D_Z) \leq 0$.
Since $(Z_K, D_K)$ is terminal, we have that $\cent_Z(Q) \subset Z_k$. 
As $Z_k$ is Cartier, we have that $a(Q, Z, Z_k+D_Z) \leq -1$, contradicting plt-ness.
\end{proof}

\begin{corollary} \label{cor: existence_extractions}
Let $(X, \Delta) \to \Spec(R)$ be a quasi-projective family of couples, such that $K_X + \Delta$ is $\mathbb{R}$-Cartier and $(X_k,\Delta_k)$ is a 3-dimensional log canonical pair.
Let $P$ be an exceptional divisorial valuation with $a(P,X,\Delta) < 0$ such that $\cent_X(P)$ dominates $\Spec(R)$.
Then, there exists a projective birational morphism $\pi \colon Z \rightarrow X$ such that the following conditions hold:
\begin{enumerate}
    \item\label{item: anti_nef} $P_Z \coloneqq \cent_Z(P)$ is a $\mathbb Q$-Cartier divisor such that $-P_Z$ is $\pi$-nef;
    \item\label{item: conn_comp} for every point $x \in X$, either $\pi^{-1}(x)$ is contained in $P_Z$ or it is disjoint from it; and
    \item\label{item: irr_fiber} $P_Z \cap \pi^{-1}_* X_k = P_Z \cap Z_k \neq \emptyset$.
\end{enumerate}
Furthermore, if $(X_K,\Delta_K)$ is klt, we may choose $\pi$ such that $\Ex(\pi)=P_Z$ and $-P_Z$ is $\pi$-ample.
\end{corollary}

\begin{proof}
First, we observe that \autoref{item: conn_comp} is a formal consequence of \autoref{item: anti_nef}, similarly to what is observed in \autoref{rmk: inv_of_adj}.
Then, \autoref{item: irr_fiber} follows from \autoref{item: conn_comp}, as $P_Z \cap Z_k \neq \emptyset$.
Thus, it suffices to show that \autoref{item: anti_nef} holds.

By \autoref{cor: inv-adjunction}, $(X, X_k+\Delta)$ is a log canonical pair.
Let $f\colon Y\to X$ be a log resolution of $(X,X_k+\Delta)$ such that $P_Y \coloneqq \cent_Y(P)$ is a divisor.
We may further assume that $\Ex(f)$ is purely divisorial and supports an $f$-ample divisor.
Write $K_Y+D=f^*(K_X+\Delta)+F$, where $D$ and $F$ are effective $\mathbb{R}$-divisors without common components.
Notice that $P_Y$ is in $\Supp(D^h)$ since $a(P,X,\Delta) < 0$.
Since $f$ is a log resolution, $(Y,Y_{k,\red}+D^h)$ is dlt.
By \autoref{t-compactmmp}, we may run a $(K_Y+Y_{k,\red}+D^h)$-MMP with scaling over $X$, which terminates with a relatively minimal model $(W,W_{k,\red}+D_W^h)$.
By construction, $W$ is $\mathbb Q$-factorial and, as $P_Y$ is in $\Supp(D^h)$, $P_W \coloneqq \cent_W(P)$ is a divisor contained in $\Supp(D_W^h)$.

Since $(X,X_k+\Delta)$ is log canonical and $(W,W_{k,\red}+D_W^h)$ is relatively minimal,
by the same arguments involving the negativity lemma as in the proof of \autoref{cor: existence_terminalisation},
we have $K_W+W_{k,\red}+D_W^h=\rho^*(K_X+X_k+\Delta)$, where $\rho \colon W \rightarrow X$ denotes the natural morphism.
Then, we have
\[
K_W+W_{k,\red}+D_W^h+a(P,X,\Delta)P_W \equiv_X a(P,X,\Delta)P_W.
\]
Then, we run a $(K_W+W_{k,\red}+D_W^h+a(P,X,\Delta)P_W)$-MMP with scaling over $X$, which terminates with a relatively minimal model $V$ by \autoref{t-compactmmp}.
Let $\sigma \colon V \rightarrow X$ denote the induced morphism.
Since it is a $(-P_W)$-MMP, $P_W$ cannot be contracted.
We denote by $P_V$ its strict transform on $V$.
Then, by construction $-P_V$ is $\sigma$-nef.
Then, to conclude that \autoref{item: anti_nef} holds, we may take $Z=V$.

If we further assume that $(X_K,\Delta_K)$ is klt, it follows that $\lfloor D_V^h \rfloor=0$.
Thus, by \cite{XX22}*{Proposition 4.1}, $-P_V$ is relatively semi-ample over $X$.
In this case, we take $Z$ to be the relatively ample model of $-P_V$ over $X$.
\end{proof}

\section{Deformations of 3-folds in positive and mixed characteristic} \label{s-defo}

We fix $R$ to be an excellent DVR with perfect residue field $k$ of characteristic $p>5$, and we assume \autoref{hyp}.

\subsection{Deformations of 3-dimensional canonical singularities} \label{ss-deformations-canonical}

We show that canonical and terminal 3-dimensional singularities are stable under small deformations.
The case of $\mathbb{Q}$-factorial terminal singularities has been proven in \cite{HW20}*{Corollary 6.7}.

\begin{proof}[Proof of \autoref{thm: def_canonical}]
By \autoref{rem: normality_family}, $X$ is normal and it has property $(S_3)$.
Write $\Delta=\sum_{i=1}^m b_i B_i$, where $b_i \in \mathbb{R}$ and $B_i$ are distinct prime divisors.
Fix a basis $(d_1,\ldots,d_l)$ of the $\mathbb{Q}$-vector space $\text{Span}_{\mathbb{Q}} (1, b_1, \dots, b_m)$ such that $d_i >0$ for all $i$ and $\sum d_i=1$.
By \cite{k-moduli}*{11.43.4}, there exists a unique collection of $\mathbb{Q}$-divisors $D_i$, not necessarily effective, such that $\Delta= \sum d_i D_i$. 
We claim that there exists a closed subset $Z \subset X$ such that $K_X+D_i$ is $\mathbb{Q}$-Cartier outside $Z$ and $\cod_{X_k}(Z \cap X_k) \geq 3$.
Indeed, by localising $X$ at a point $p\in X_k$ such that $\codim_{X_k}(\overline{\{p\}})=2$, we have two possibilities: either $\Delta_p=0$ or $\Delta_p \neq 0$.
Here $\Delta_p$ (resp.~$D_{i,p}$) denotes the restriction of $\Delta$ (resp.~$D_i$) to $\Spec(\mathcal{O}_{X,p})$.
If $\Delta_p=0$, we have $D_{i,p}=0$ by the uniqueness of the representation and $K_X$ is thus Cartier in a neighbourhood of $p$ by \cite{stacks-project}*{\href{https://stacks.math.columbia.edu/tag/0BJJ}{Tag 0BJJ}} and the fact that canonical surface singularities are Gorenstein (see \cite{kk-singbook}*{Theorem 2.29}).
If $\Delta_p \neq 0$ holds, by \cite{kk-singbook}*{Theorem 2.29}, $X_k$ and therefore $X$ are regular at $p$ and hence $K_X+D_i$ is $\mathbb Q$-Cartier around $p$.
As $K_{X_k}+\Delta_k = \sum d_i (K_{X_k}+D_{i,k})$ is $\mathbb{R}$-Cartier, we have by \cite{k-moduli}*{Claim 11.43.2.(a)} that $K_{X_k}+D_{i,k}$ is $\mathbb{Q}$-Cartier and we denote by $n_i$ its Cartier index
and by $m_i$ the least common multiple between $n_i$ and the Weil index of $D_i$.
As $X_k$ has property $(S_3)$ by \cite{BK20}*{Theorem 3 and 19}, we can apply \cite{dFH11}*{Proposition 3.1} to deduce that $m_i(K_X+D_i)$ is Cartier, thus showing that $K_X+\Delta$ is $\mathbb{R}$-Cartier.
Note that, if $\Delta$ is a $\mathbb Q$-divisor, we have $l=1$ and $K_X+\Delta=K_X+D_1$.
Then, under this assumption, if $K_{X_k}+\Delta_k$ has Cartier index $I_2$, \cite{dFH11}*{Proposition 3.1} implies
that $\mathrm{lcm}(I_1,I_2)(K_X+\Delta)$ is also Cartier, where $I_1$ denotes the Weil index of $\Delta$.
Finally, by inversion of adjunction \autoref{cor: inv-adjunction}, we know that $(X, X_k+\Delta)$ is log canonical (resp.~ plt).

First, we show that $(X,\Delta)$ is canonical.
Arguing by contradiction, we suppose that there exists an exceptional valuation $P$ over $X$ such that $a(P, X, \Delta)<0$.  
By \autoref{lem: lemma_trick}, $\cent_X(P)$ dominates $\Spec(R)$.
Then, let $\pi \colon Z \rightarrow X$ be a birational morphism as in \autoref{cor: existence_extractions}, and set $\pi^*(K_X+X_k+\Delta)=K_Z+\pi^{-1}_*X_k+\Delta_Z$.
Since $Z_k$ is a Cartier divisor, $P_Z \coloneqq \cent_Z(P)$ is $\mathbb Q$-Cartier, and $P_Z \cap Z_k \subset \pi^{-1}_*X_k$, it follows that $P_Z \cap Z_k$ is a divisor in $\pi^{-1}_*X_k$.
Furthermore, by upper semi-continuity of the fibre dimension, it is exceptional for $\pi^{-1}_* X_k \rightarrow X_k$.
Then, by adjunction, the coefficients of $\textup{Diff}_{(\pi^{-1}_*X_k)^\nu}(\Delta_Z)$ along the prime components of $P_Z \cap \pi^{-1}_*X_k$ are positive.
Furthermore, these components are $\pi_k^\nu$-exceptional, where $\nu \colon (\pi^{-1}_*X_k)^\nu \to \pi^{-1}_*X_k$ is the normalisation of $\pi^{-1}_*X_k$ and $\pi_k^\nu \colon (\pi^{-1}_*X_k)^\nu \rightarrow X_k$ is the induced morphism.
Then, since $((\pi^{-1}_*X_k)^\nu,\textup{Diff}_{(\pi^{-1}_*X_k)^\nu}(\Delta_Z))$ is crepant to $(X_k,\Delta_k)$, it follows that $(X_k,\Delta_k)$ is not canonical, and we get the sought contradiction.

Now, assume that $(X_k,\Delta_k)$ is terminal.
Then, $(X,\Delta)$ is canonical by the above paragraph.
Assume by contradiction that $(X,\Delta)$ is not terminal, and let $\phi \colon (W,D_W) \rightarrow (X,\Delta)$ be a terminalisation as in \autoref{cor: existence_terminalisation}.
Then, $\phi$ extracts some exceptional divisorial valuation $Q$ with $a(Q,X,\Delta)=0$.
By \autoref{lem: lemma_trick}, $W_k$ coincides with the strict transform of $X_k$ and $Q_W \coloneqq \cent_W(Q)$ dominates $\Spec(R)$.
Then, we may argue as in the previous paragraph to conclude that the existence of $Q_W$ contradicts the fact that $(X_k,\Delta_k)$ is terminal.

We are left to check that $(X_{\overline{K}}, \Delta_{\overline{K}})$ is canonical (resp.~ terminal). 
In this case, the same proof of \autoref{cor: def_canonical} works.
\end{proof}

\begin{remark}
    We recall that one cannot generalise \autoref{thm: def_canonical}
    to the case of klt singularities as shown by the examples in \cite{Ish18}*{Examples 9.1.7 and 9.1.8}.
\end{remark}

In \cite{SS22}, Sato and Takagi showed that, in characteristic 0, klt and log canonical singularities deform in a valuative sense (see \cite{SS22}*{Definition 2.6}).
In particular, if the canonical divisor of the generic fiber of the deformation is $\mathbb Q$-Cartier, then the generic fiber is klt, with no assumptions on the canonical divisor of the total space.
In the following, we show an analogue of this result for 3-dimensional klt singularities in positive and mixed characteristics.

\begin{theorem} \label{thm: def_klt}
	Let $R$ be an excellent DVR with perfect residue field $k$ of characteristic $p>5$.
	Let $(X, \Delta) \to \Spec(R)$ be a family of $3$-dimensional couples.
	If $(X_k, \Delta_k)$ is klt, then there exists a small morphism $X'_{{K}} \rightarrow X_K$ such that $(X'_{{K}},\Delta'_K)$ is klt, where $\Delta'_K$ denotes the strict transform of $\Delta_K$.
    In particular, if $K_{X_K}+\Delta_K$ is $\mathbb R$-Cartier, $(X_K,\Delta_K)$ is geometrically klt.
 \end{theorem} 

\begin{proof}
    By \autoref{rem: normality_family}, $X$ is normal.
    Let $\pi \colon (X',X'_{k,\red}+\Delta') \rightarrow (X,X_k+\Delta)$ be a dlt modification as in \autoref{cor: existence-dlt-modificatin}.
    We observe that, if $\pi$ is small, then the claim follows.

    Thus, in the following, we are left with showing that $\pi$ is small.
    We define $S'$ to be the normalisation of $\pi^{-1}_*X_k$, and we let $(S',\Delta_{S'})$ denote the pair structure induced by $(X',X'_{k,\red}+\Delta')$.
    As $(X',X'_{k,\red}+\Delta')$ is dlt, $S' \rightarrow \pi^{-1}_*X_k$ is a universal homeomorphism by \cite{HW20}*{Lemma 2.1}.

    Since $(X_k,\Delta_k)$ is klt, by localising at codimension 1 points of $X_k$, we may apply \cite{kk-singbook}*{Theorem 2.29} to conclude that there exists a big open subset $U_k \subset X_k$ such that $(U_k,\Delta_k|_{U_k})$ is canonical.
    Then, by \autoref{thm: def_canonical}, it follows that $S' \rightarrow X_k$ is crepant over $U_k$.
    Indeed, let $U$ be an open subset of $X$ such that $U \cap X_k=U_k$.
    By \autoref{thm: def_canonical}, $(U,\Delta|_U)$ is canonical.
    Then, by \autoref{def: dlt-modification} and the negativity lemma \cite{7authors}*{Lemma 2.16}, $\pi$ is crepant and small over $U$.
    In particular, the special fibre of $U' \rightarrow \Spec(R)$ is irreducible, where $U'$ denotes the preimage of $U$ via $\pi$.
    Hence, $U'_k$ is the pull-back of $U_k$, and it follows that $(U',X'_{k,\red}|_{U'}+\Delta'|_{U'})$ is crepant to $(U,U_k+\Delta|_U)$.
    Then, by adjunction, $(U'_k,\Delta_{S'}|_{U'_k})$ is crepant to $(U_k,\Delta_k|_{U_k})$.

    In particular, it follows that $\Delta_{S'}$ is the sum of the strict transform of $\Delta_k$ and some $\pi_{S'}$-exceptional divisors, where we set $\pi_{S'} \colon S' \rightarrow X_k$.
    Then, as $(K_{S'}+\Delta_{S'})$ is $\pi_{S'}$-nef by construction, by the negativity lemma it follows that $K_{S'}+\Delta_{S'} \leq \pi^*_{S'}(K_{X_k}+\Delta_k)$.
    Since $(X_k,\Delta_k)$ is klt, then so is $(S',\Delta_{S'})$.

    Therefore, we conclude that $\pi^{-1}_*X_k$ does not intersect any $\pi$-exceptional divisor.
    Indeed, by definition of dlt modification, every $\pi$-exceptional divisor has coefficient 1 in $X'_{k,\red}+\Delta'$.
    Thus, by adjunction and $\mathbb{Q}$-factoriality of $X'$, any such divisor would contribute with a divisor with coefficient 1 in $\Delta_{S'}$, which is impossible.
    Then, it follows that $X'_k=\pi^{-1}_*X_k$.
    Similarly, no exceptional divisor can dominate $\Spec(R)$, as it would otherwise intersect $X'_k=\pi^{-1}_*X_k$.
    Then, the claim follows, and we may conclude.

    We are left to check that $(X_{\overline{K}}, \Delta_{\overline{K}})$ is klt. 
    In this case, the same proof of \autoref{cor: def_canonical} works by replacing canonical with klt.
\end{proof}

For completeness we include the 2-dimensional case (for the case of strongly $F$-regular, cf. \cite{ST21}).

\begin{theorem}\label{thm: def_klt_surfaces}
    Let $R$ be an excellent DVR
    and let $f \colon (X,\Delta) \to \Spec(R)$ be a family of 2-dimensional couples.
    If $(X_k, \Delta_k)$ is klt, then so is $(X_K,\Delta_K)$.
    Moreover, if $K_{X_k}+\Delta_k$ is geometrically klt, then so is $(X_K,\Delta_K)$. 
\end{theorem}

\begin{proof}
    Using the existence of (not-necessarily $\mathbb{Q}$-factorial) dlt modifications \cite{k-notQfact}*{Theorem 9} for threefolds, the proof proceeds as in the proof of \autoref{thm: def_klt}, and it uses the fact that every small morphism between normal surfaces is an isomorphism.

    The last assertions follows from the same proof of \autoref{cor: def_canonical}.
\end{proof}

\subsection{Applications to families of Fano varieties and minimal models} \label{ss-deformations-Fano}

We present applications of our results to families and liftings of 3-dimensional terminal Fano varieties and minimal models.

\begin{proposition} \label{cor: lifting_Fano}
    Let $(X,\Delta) \to \Spec(R)$ be a projective family of 3-dimensional couples where $\Delta$ is a $\bQ$-divisor.
    Assume that $(X_k,\Delta_k)$ is a weak Fano pair with terminal (resp. ~canonical) singularities. Then $(X,\Delta)$ is terminal (resp. ~canonical) and $(X_K,\Delta_K)$ is a weak Fano pair with geometrically terminal (resp. ~canonical) singularities. Moreover, $-(K_X+\Delta)$ is big and semi-ample over $\Spec(R)$.
\end{proposition}

\begin{proof}
    As $(X, \Delta)$ is terminal and $(X_K, \Delta_K)$ is geometrically terminal (resp.~ canonical) by  \autoref{thm: def_canonical}, $-(K_{X_K}+\Delta_K)$ is $\mathbb{Q}$-Cartier.
    We are left to check $-(K_{X_K}+\Delta_K)$ is big and nef.
    Let $C_K$ be an effective integral curve on $X_K$, which specialises to an effective (possibly non-integral) curve $C_k \subset X_k$.
    As $-(K_{X_K}+\Delta_K) \cdot_K C_K = -(K_{X_k}+\Delta_k) \cdot_k C_k \geq 0$, we conclude $-(K_{X_K}+\Delta_K)$ is nef.
    Moreover, it is big as $(-K_{X_K}-\Delta_K)^3=(-K_{X_k}-\Delta_k)^3>0$.
    By the base-point free theorem \cite{DW22}*{Theorem 1.4} $-(K_{X_K}+\Delta_K)$ and $-(K_{X_k}+\Delta_k)$ are semi-ample, thus the last assertion follows from \cite{Wit21}*{Theorem 1.2}.
\end{proof}

\begin{corollary} \label{cor: fano}
    Let $X \to \Spec(R)$ be a projective family of 3-dimensional varieties.
    Assume that $X_k$ is a weak Fano variety with canonical singularities.
    Moreover, suppose that $\characteristic(R)=0$.
    Then, we have $1/330 \leq \mathrm{vol}(X_k,-K_{X_k}) \leq 324$, $h^0(X_k,-6K_{X_k}) > 0$, $h^0(X_k,-8K_{X_k}) > 1$, and $|-rK_{X_k}|$ is birational for every $r \geq 97$.
\end{corollary}

\begin{proof}
By \autoref{cor: lifting_Fano}, $K_X$ is $\mathbb{Q}$-Cartier and therefore $(-K_{X_k})^3=(-K_{X_K})^3$.
The statements regarding the bounds on the volume and the dimension of the space of sections now follow immediately from \autoref{cor: lifting_Fano}, upper semi-continuity of sections, and \cites{MR2456276,JZ21}.

Now, let $r \geq 97$ be an integer.
By \cite{MR3544286}, $|-rK_{X_{\overline{K}}}|$ induces a birational map.
Since the canonical divisor is defined over $K$ and the extension $\overline{K}:K$ is faithfully flat, $|-rK_{X_{K}}|$ induces a birational map.
Then, $|-rK_{X}|$ induces a birational map of $ \varphi \colon X \dashrightarrow Z$ over $\Spec(R)$.
Let $U\subset X$ be an open containing the generic point of $X_k$ and such that $\varphi\vert_U$ is a morphism.
As the degree of a morphism is constant, we conclude the induced map $\varphi|_{X_k}$ is birational.
By upper semi-continuity, $\varphi|_{X_k}$ is defined by a sub-linear system of $|-rK_{X_k}|$, thus showing that $|-rK_{X_k}|$ also induces a birational map.

We observe that some of the references are phrased for weak Fano varieties with terminal singularities;
we observe that they also apply in our setup by taking a crepant terminalisation as in \autoref{cor: existence_terminalisation}; see also \cite{MR3544286}*{Remark 1.9}.
\end{proof}

\begin{corollary}\label{cor: lift_MM_char0}
        Let $(X,\Delta)\to \Spec (R)$ be a projective family of 3-dimensional couples.
        Assume that $\characteristic(R)=0$ and $\Delta$ is a $\bQ$-divisor.
        Assume that
    \begin{enumerate}
        \item $(X_k,\Delta_k)$ is a terminal (resp.~ canonical) minimal model; or
        \item $(X_k,\Delta_k)$ is a klt minimal model and $K_X+\Delta$ is $\bQ$-Cartier.
    \end{enumerate}
    Then $(X_k,\Delta_k),(X_{\overline{K}},\Delta_{\overline{K}})$ and $(X,\Delta)$ are good minimal models, with terminal (resp. canonical) singularities in case (a), and with klt singularities in case (b). In particular $\kappa(K_{X_k}+\Delta_k)=\kappa(K_{X_K}+\Delta_K)$.
\end{corollary}
\begin{proof}
    We first assume (a).
    By~\autoref{thm: def_canonical}, we have that $(X_{\overline{K}},\Delta_{\overline{K}})$ and $(X,\Delta)$ are both terminal (resp.~ canonical).
    In particular, $K_X+\Delta$ is $\bQ$-Cartier.
    By arguing as in the proof of \autoref{cor: lifting_Fano}, we have that $K_{X_K}+\Delta_K$ is nef as well, hence it is also semi-ample by the abundance theorem for 3-folds in characteristic zero. Recall that, in order to show that $K_{X_k}+\Delta_k$ is semi-ample, by \cite{MR97}*{Corollary 1} it is enough to show it is abundant (i.e., $\kappa(K_{X_k}+\Delta_k)=\nu(K_{X_k}+\Delta_k)$) by finite generation of the canonical ring; see \cite{WalFG}*{Theorem 1.2}.
    As intersection numbers are constant in a flat family, we have $\nu(K_{X_k}+\Delta_k)=\nu(K_{X_K}+\Delta_K)$ and by upper semi-continuity of global sections we have the following chain of inequalities
    \[
    \kappa(K_{X_k}+\Delta_k) \geq \kappa(K_{X_K}+\Delta_K)=\nu(K_{X_K}+\Delta_K)=\nu(K_{X_k}+\Delta_k),
    \] 
    showing that $K_{X_k}+\Delta_k$ is abundant.
    By \cite{Wit21}*{Theorem 1.2} we also have that $K_X+\Delta$ is semi-ample over $R$, hence $(X,\Delta)$ is a terminal (resp.~ canonical) good relative minimal model over $R$.

    In case (b), we argue in the same way, replacing \autoref{thm: def_canonical} by \autoref{cor: plt-inv-adj} (resp. \autoref{cor: inv-adjunction}).
\end{proof}

\begin{corollary}\label{cor: lift_MM_charp}
    Let $(X,\Delta)\to \Spec (R)$ be a projective family of 3-dimensional couples.
    Assume that $\characteristic(R)=p>5$ and $\Delta$ is a $\mathbb Q$-divisor.
    Assume that
    \begin{enumerate}
        \item $(X_k,\Delta_k)$ is a terminal (resp.~ canonical) minimal model, $\nu(K_{X_k}+\Delta_k) \neq 1$; or
        \item $(X_k,\Delta_k)$ is a klt minimal model, $K_X+\Delta$ is $\bQ$-Cartier, and $\nu(K_{X_k}+\Delta_k)\neq 1$.
    \end{enumerate}
    Then $(X_k,\Delta_k),(X_{\overline{K}},\Delta_{\overline{K}})$ and $(X,\Delta)$ are good minimal models, with terminal (resp. canonical) singularities in case (a), and with klt singularities in case (b). In particular $\kappa(K_{X_k}+\Delta_k)=\kappa(K_{X_K}+\Delta_K)$.
\end{corollary}
\begin{proof}
    The proof follows the exact same argument as that of \autoref{cor: lift_MM_char0}.
    Note that, when $\nu(K_{X_K}+\Delta_K)\neq 1$, the abundance theorem for $(X_K,\Delta_K)$ holds by \cite{WalFG}*{Theorem 1.3}, \cite{DW22}*{Theorem 1.4}, and \cite{WitCBF}*{Theorem 3}.
\end{proof}

\begin{corollary} \label{cor: volume_gen_type}
    Let $X \to \Spec(R)$ be a projective family of 3-dimensional varieties.
    Assume that $X_k$ is canonical and $K_{X_k}$ is big.
    Moreover, suppose that $\characteristic(R)=0$.
    Then, we have $\mathrm{vol}(X_k,K_{X_k}) \geq 1/1680$ and $|rK_{X_k}|$ is birational for $r \geq 57$.
\end{corollary}

\begin{proof}
By \autoref{thm: def_canonical}, $X$ is canonical and $(X,X_k)$ is plt.
By \cite{XX22}*{Theorem 1.1}, we may run a $(K_X+X_k)$-MMP with scaling over $\Spec(R)$, which terminates with a model $Y$.
By \cite{Bri22}*{Lemma 2.18}, this MMP restricts to a sequence of $K_{X_k}$-negative contractions on the special fibre.
Since $(X,X_k)$ is plt, then so is $(Y,Y_k)$.
Then, $Y_k$ is normal by \cite{HW20}*{Lemma 6.3} and hence canonical and crepant to $X_k$.
Since $K_{X_k}$ is big, then so is $K_{Y_k}$.
Thus, $Y$ is a minimal model.
Since the pluricanonical ring of $X_k$ coincides with the one of $Y_k$, up to relabelling, we may then assume that $X_k$ is a minimal model.
Then, by \autoref{cor: lift_MM_char0}, also $X_K$ is a minimal model of general type.
Since $K$ has characteristic 0, by \cites{MR3357178,MR3824565}, we have $K_{X_K}^3=\mathrm{vol}(X_K,K_{X_K}) \geq 1/1680$ and $|rK_{X_K}|$ is birational for $r \geq 57$.
Then, as $K_{X_K}^3=K_{X_k}^3$, we have $\mathrm{vol}(X_k,K_{X_k}) \geq 1/1680$.

The last part of the statement follows from \cite{MR3357178} arguing as in the last step of the proof of \autoref{cor: fano}.
\end{proof}

We conclude by showing that being a weak Fano 3-fold or a minimal model with terminal or canonical singularities is an open condition in the Zariski topology. 

\begin{lemma}\label{lem: spread_out}
Let $S$ be a Dedekind domain whose closed points have perfect residue field of characteristic $p>5$ and let $(X, \Delta) \to \Spec(S)$ be a family of 3-dimensional couples, where $\Delta$ is a $\mathbb{Q}$-divisor.
If $(X_K, \Delta_K)$ is geometrically terminal (resp.~ canonical), then there exists a Zariski open set $U \subset \Spec(S)$ such that $(X_u, \Delta_u)$ is terminal  (resp.~ canonical) for all $u\in U$.
\end{lemma}

\begin{proof}
    Without loss of generality, we can suppose $X \to \Spec(S)$ is a morphism of finite type. We show the terminal case, as the canonical one is proven similarly.
    First, suppose that $\characteristic(S)=0$ and let $f \colon Y \to (X, \Delta)$ be a log resolution.
    Note that, as $K_{X_K}+\Delta_K$ is $\bQ$-Cartier, we may assume that so is $K_X+\Delta$, after possibly shrinking $\Spec(S)$.
    Write $K_Y + D= f^*(K_X+\Delta) +F$, where $D$ and $F$ are effective divisors without irreducible components in common. 
    As $(X_K, \Delta_K)$ is terminal, then the $f$-exceptional divisors with non-positive discrepancy are vertical over $S$ and thus we can suppose that the support of $F$ coincides with the divisorial part of $\mathrm{Ex}(f)$.
    As $K$ has characteristic 0, $(Y_K, F_K)$ is geometrically snc.
    Hence, there exists an open set $U \subset \Spec(S)$ such that for all $u \in U$ the morphism $Y_u \to (X_u, \Delta_u)$ is a log resolution by 
    \cite{EGAIV.III}*{Théorème 12.1.6}, thus showing $(X_u, \Delta_u)$ is terminal. 

    Suppose that $\characteristic(S)=p>0$. 
    Let $X_{K^{1/p^\infty}}$ be the base change to the perfect closure of $K$.
    Let $g_\infty \colon Z_\infty \to (X_{K^{1/p^\infty}}, \Delta_{K^{1/p^\infty}})$ be a log resolution. Note $Z_{\infty}$ is geometrically regular as $K^{1/p^\infty}$ is perfect.
    By the hypothesis of finite type, there exists $e>0$ such that there is a log resolution $Z_e \to (X_{K^{1
    /p^e}}, \Delta_{K^{1/p^e}}) $, such that $Z_e \times_{K^{1/p^e}} K^{1/p^\infty} \simeq Z_{\infty} $ and thus $Z_e$ is geometrically regular. 
    Consider the base change $X^{(e)} \to \Spec(S^{1/p^e})$. 
    Up to further localising on the base, we can suppose that $Z_e
    \to (X_{K^{1/p^e}}, \Delta_{K^{1/p^e}})$ spreads to a log resolution $Y \to (X^{(e)}, \Delta^{(e)})$. 
    Following the same proof as in characteristic 0, we deduce that $(X^{(e)}_{t},\Delta^{(e)}_t)$ is terminal for general $t \in  \Spec(S^{1/p^e})$.
    As the residue fields of closed points are perfect, this implies that also $(X_u, \Delta_u)$ is terminal for general $u \in \Spec(S)$.
\end{proof}

\begin{corollary}\label{cor: def_Fano}
Let $S$ be a Dedekind domain whose closed points have perfect residue field of characteristic $p>5$ and let $(X, \Delta) \to \Spec(S)$ be a projective family of couples.

Let $s \in \Spec(S)$ be a closed point such that $(X_s, \Delta_s)$ is a 3-dimensional weak Fano variety with terminal (resp.~ canonical) singularities.
Then, there exists a Zariski open set $U \subset \Spec(S)$ such that $(X_u, \Delta_u)$ is a weak Fano with terminal (resp.~ canonical) singularities for all $u \in U$ . 
\end{corollary}

\begin{proof}
    Consider the localisation $R_s \coloneqq \mathcal{O}_{S, s}$ at $s \in S$, which is a DVR with fraction field $K=\Frac(S)$.
    By \autoref{cor: lifting_Fano},  $(X_K, \Delta_K)$ is geometrically terminal (resp.~ canonical) and $-(K_{X_{R_s}}+\Delta_{R_s})$ is big and semi-ample. 
   Therefore, there exists an open set $U$ containing $s$ for which all fibres $u \in U$ the pair $(X_u, \Delta_u)$ is terminal (resp.~ canonical) by \autoref{lem: spread_out} and $-(K_{X_u}+\Delta_u)$ is big and semi-ample. 
\end{proof}

\begin{corollary}\label{cor: minimal_models}
Let $S$ be a mixed characteristic Dedekind domain whose closed points have perfect residue field $p >5$ and let $(X, \Delta) \to \Spec(S)$ be a projective family of couples.

Let $s \in \Spec(S)$ be a closed point such that $(X_s, \Delta_s)$ is a 3-dimensional minimal model with terminal (resp.~canonical) singularities.
Suppose one of the following holds:
\begin{enumerate}
    \item $\characteristic(R)=0$; or
    \item $\characteristic(R)=p>5$  and $\nu(K_{X_s}+\Delta_s) \neq 1$.
\end{enumerate}
Then, there exists an open set $U \subset \Spec(S)$ such that $(X_u, \Delta_u)$ is a good minimal model with terminal (resp.~ canonical) singularities for all $u \in U$. 
\end{corollary}

\begin{proof}
    We can repeat the same proof of \autoref{cor: def_Fano} using \autoref{cor: lift_MM_char0} and \autoref{cor: lift_MM_charp} .
\end{proof}

\section{Deformations of non-klt canonical singularities in characteristic 0}\label{s-char_0}

Throughout this section, we fix $R$ to be an excellent DVR with residue field of characteristic 0.

In \cite{dFH11}, de Fernex and Hacon showed that, in characteristic 0, small deformations of canonical pair singularities that are also klt remain canonical; see \cite{dFH11}*{Proposition 3.5.(b)}.
In particular, they did not address the case of canonical pair singularities that are not klt.
Also, they assume that $R$ is the local ring of a smooth curve, as opposed to a more general DVR.
The arguments of \autoref{thm: def_canonical} apply to treat the non-klt case and the case of real coefficients in arbitrary dimension in characteristic 0.
In light of \cite{LM22}, the case of general $R$ is also addressed.
Thus, for the sake of completeness, we include the following statement.

\begin{theorem}[cf.~\cite{dFH11}*{Proposition 3.5.(b)}]\label{theorem: char_0}
Let $(X,\Delta) \to \Spec(R)$ be a family of couples such that $(X_k, \Delta_{X_k})$ is a canonical (resp.~ terminal) pair.
Then $(X, \Delta)$ is a canonical (resp.~ terminal) pair.
Moreover, $(X_{\overline{K}}, \Delta_{\overline{K}})$ is canonical (resp.~ terminal) and therefore $(X, \Delta)$ is a locally stable family of pairs.
\end{theorem}

\begin{proof}
    The proof of \autoref{thm: def_canonical} applies verbatim since the needed statements from the MMP used in \autoref{s-MMP-input} are well established in every dimension in characteristic 0 by \cite{LM22}.
\end{proof}

\section{Jumps of the Cartier index of $K_X$} \label{s-example}

In this section, we show that the index of $K_X$ can jump for families of canonical singularities in equi-characteristic $p=2$. 

\begin{example}\label{jump:imperfect}
Let $k$ be an algebraically closed field of characteristic 2 and let $K \coloneqq k(t)$ be a purely transcendental extension.
Let $S$ be the del Pezzo surface over $K$ constructed by Schr\"{o}er in \cite{Sch07}.
We recall the properties of $S$ that we will need:
\begin{enumerate}
    \item $S$ is geometrically integral and not geometrically normal by \cite{Sch07}*{Theorem 8.1};
    \item $S$ has canonical singularities \cite{Sch07}*{Theorem 8.2};
    \item \label{itm: vanishing} for any ample line bundle $\mathcal{H}$ on $S$, $H^1(S, \mathcal{H})=0$ by \cite{Sch07}*{Proposition 5.1} and flat base change; and
    \item \label{itm: non-vanishing} $\Pic_{S/K}^{0} \simeq \mathbb{G}_{a,K}$ by \cite{Sch07}*{Equation 14}, in particular $h^1(S,\cO_S)=1$ holds.
\end{enumerate}
Consider $\pi \colon Y \coloneqq S \times \mathbb{G}_{a,K} \to \mathbb{G}_{a,K} \simeq \Spec K[u]$ together with the Poincar\'{e} line bundle $\mathcal{L}$ \cite{FGA}*{Exercise 9.4.3}.
We denote by $\tau \colon Y \rightarrow S$ the other projection.
Note that $\mathcal{L}_0=\mathcal{O}_S$ while $\mathcal{L}_u$ is a non-trivial $2$-torsion line bundle on $S$ for all $u\in\mathbb{G}_{a,K} \setminus \lbrace 0\rbrace$.
Consider the line bundle $\mathcal{M} \coloneqq \tau^*\omega_S^{\vee} \otimes \mathcal{L}$ and let $f \colon X \coloneqq  \Spec_{\mathbb{G}_{a,K}} \bigoplus_{n \geq 0} \pi_*(\mathcal{M}^{\otimes n}) \to \mathbb{G}_{a,K}$ be the relative cone over $\mathbb{G}_{a,K}$. 
We claim that, for all $u\in\mathbb{G}_{a,K}$, the fibre $X_u$ is the cone over $S$ with respect to the ample line bundle $\omega_S^{\vee}\otimes \mathcal{L}_u$.
This is equivalent to the surjectivity of the restriction map
\[
H^0(Y,\mathcal{M}^{\otimes n}) \otimes k(u) \to H^0(Y_u,\mathcal{M}_u^{\otimes n}) \simeq H^0(S, \mathcal{M}_u^{\otimes n})
\]
for all $u \in\mathbb{G}_{a,K}$ and all $n \geq 0$. As $Y$ is a product, this is immediate for $n=0$.
By Grauert's theorem \cite{Ha77}*{Corollary 12.9} and cohomology and base change \cite{Ha77}*{Theorem 12.11}, to show the surjectivity for $n \geq 1$, it is enough to show the vanishing of $H^1(S,\mathcal{M}_u^{\otimes n})$. 
As $\mathcal{M}_u^{\otimes n}$ is ample, we have the desired vanishing by \autoref{itm: vanishing}.

By \cite{kk-singbook}*{Lemma 3.1}, $X \to \mathbb{G}_{a,K}$ is a 3-dimensional family such that 
$X_0$ and $X$ are canonical.
Moreover $K_{X_0}$ is Cartier and $K_{X_u}$ is $\mathbb{Q}$-Cartier of index $2$ for all $u\in\mathbb{G}_{a, K} \setminus \{0\}$ by \cite{kk-singbook}*{Lemma 3.14}. 


A word of caution: it is true that $K_{X_0}$ is Cartier but $X_0$ is not $(S_3)$, in particular $X_0$ is not Gorenstein so we cannot apply \cite{stacks-project}*{\href{https://stacks.math.columbia.edu/tag/0BJJ}{Tag 0BJJ}}.
The fact that $X_0$ is not $(S_3)$ can be shown as follows.
Since $X_0$ is a 3-dimensional cone, its structure sheaf is $(S_3)$ at the vertex if and only if it is Cohen--Macaulay.
Thus, it suffices to show that $X_0$ is not Cohen--Macaulay.
Then, this follows by \cite{kk-singbook}*{Corollary 3.11} and property \autoref{itm: non-vanishing}.
\end{example}

The following example shows \autoref{t-example}.

\begin{example}\label{jump:perfect}
    Using the ``spreading out'' technique, we construct an example of a deformation of a 4-dimensional, non-$(S_3)$ canonical singularity over a perfect field of characteristic 2, $\varphi\colon\cX\to \bA^1_{k}\coloneqq\Spec(k[u])$, where $K_{\mathcal{X}_u}$ has index $2$ for $u\in\mathbb{A}^1_k\setminus\lbrace 0\rbrace$, but $K_{\mathcal{X}_0}$ is Cartier. 
    
    By spreading out the example of \cite{Sch07}, there exists a non-empty open set $T \subset \mathbb{A}^1_k=\Spec (k[t])$ such that 
    \begin{enumerate}
    \item there is a projective flat morphism $s \colon \mathcal{S} \to T$;
    \item the total space $\mathcal{S}$ has canonical singularities and $\mathcal{S}_K \simeq S$;
    \item $\omega_{S}^{\vee}$ is ample over $T$ and $\rho(\mathcal{S}/T)=1$; and
    \item $\Pic^0_{\mathcal{S}/T} \simeq \mathbb{G}_{a, T}  \coloneqq  \mathbb{G}_{a,k} \times_k T=\Spec(k[u]) \times_k T$.
    \end{enumerate} 
    Consider the Poincar\'{e} line bundle $\mathcal{P}$ over $\pi \colon \mathcal{S} \times_T \mathbb{G}_{a,T} \to \mathbb{G}_{a,T}$, which is a spreading out of the Poincar\'{e} bundle $\mathcal{L}$ of \autoref{jump:imperfect}. Consider the natural projection $\psi\coloneqq \cS\times_T\mathbb{G}_{a,T}\to \mathbb{G}_{a,k}=\Spec(k[u])$: then, for all $u \in \mathbb{G}_{a,k} \setminus \{0\}$, we have that the base change of the Poincar\'e line bundle $\mathcal{P}_u$ is 2-torsion but not trivial. 
    We denote the projection onto the first factor by $\tau \colon \mathcal{S} \times_T \mathbb G _{a,T} \rightarrow \mathcal{S}$ and we define $\mathcal{N} \coloneqq  \mathcal{P} \otimes \tau^*{\omega_{\mathcal S}^{\vee}}$.
    Consider the relative cone $\mathcal{X} \coloneqq \Spec_{\mathbb{G}_{a,T}} \bigoplus_{n \geq 0} \pi_* (\mathcal{N}^{\otimes n}) \to \mathbb{G}_{a,T}$ and denote by $\varphi \colon \mathcal{X} \to \mathbb{G}_{a,k}$ the natural projection.
    Again by cohomology and base change, we have $\mathcal{X}_0= \Spec_T \bigoplus_{n \geq 0} \pi_* \mathcal{O}_{\mathcal{S}}( -nK_{\mathcal{S}}) \to T \to \Spec(k)$ is a canonical singularity of dimension 4 and with Cartier index 1 by \cite{kk-singbook}*{Lemma 3.1}, while $\mathcal{X}_{u}$ has Cartier index 2 for $u \neq 0$ by \autoref{jump:imperfect}.
\end{example}

\bibliographystyle{amsalpha}
\bibliography{refs}
	
\end{document}